\documentclass[10pt]{amsart}
\usepackage{amsmath, amsfonts, fullpage, amsthm, mathpazo, amssymb, color, mathrsfs}

\newcommand{\R}{\mathbf{R}}

\newcommand{\N}{\mathbf{N}}

\newcommand{\E}{\mathbf{E}}

\newcommand{\mcL}{\mathcal{L}}
\newcommand{\mcE}{\mathcal{E}}

\renewcommand{\E}{\mathbf{E}}
\renewcommand{\P}{\mathbf{P}}

\numberwithin{equation}{section}
\theoremstyle{definition}

\theoremstyle{definition}
\newtheorem{remark}[equation]{Remark}
\theoremstyle{definition}

\theoremstyle{definition}

\theoremstyle{lemma}
\newtheorem{assumption}[equation]{Assumption}
\theoremstyle{lemma}
\newtheorem{lemma}[equation]{Lemma}
\theoremstyle{theorem}
\newtheorem{theorem}[equation]{Theorem}
\theoremstyle{proposition}

\theoremstyle{corollary}
\newtheorem{corollary}[equation]{Corollary}

\theoremstyle{corollary}
\theoremstyle{definition}
\newtheorem{example}[equation]{Example}
\theoremstyle{example}

\theoremstyle{proposition}

\title{The Small-Mass Limit for Langevin Dynamics with Unbounded Coefficients and positive friction}
\author{David P. Herzog}
\author{Scott Hottovy}
\author{Giovanni Volpe}
\begin{document}
\maketitle

\begin{abstract}
A class of Langevin stochastic differential equations is shown to converge in the small-mass limit under very weak assumptions on the coefficients defining the equation.  The convergence result is applied to physically realizable examples where the coefficients defining the Langevin equation grow unboundedly either at a boundary, such as a wall, and/or at the point at infinity. 
\end{abstract}

\section{Introduction}

Let $\mathcal{X}\subset \R^n$ be non-empty, open.  We study the following stochastic differential equation
\begin{align}
\label{eqn:main}
dx^m(t) &= v^m(t) \, dt \\
\nonumber m\, dv^m(t)& = [F(x^m(t)) -\gamma(x^m(t)) v^m(t)] \, dt  + \sigma(x^m(t)) \, dB(t)
\end{align}
where $F:\mathcal{X}\rightarrow \R^n$, $\gamma:\mathcal{X}\rightarrow \R^{n\times n}$, $\sigma : \mathcal{X}\rightarrow \R^{n\times k}$, $m>0$ is a constant and $B(t)=(B_1(t), \ldots, B_k(t))^T$ is a $k$-dimensional Brownian motion defined on a probability space $(\Omega, \mathcal{F}, \P)$.  Relation~\eqref{eqn:main} is the standard form of Newton's equation for the position $x^m(t)$ of a particle of mass $m$ subject to thermal fluctuations ($\sigma(x^m(t)) \, dB_t$), friction ( $-\gamma(x^m(t)) v^m(t) \, dt$), and a force ($F(x^m(t))\, dt$).  

The goal of this note is to strengthen the main result in \cite{HMVW_14} concerning the small-mass limit of the position $x^m(t)$.  In essence, provided the friction matrix $\gamma(x)$ is positive-definite for each $x\in \mathcal{X}$, our main result shows that we can still extract convergence of $x^m(t)$ as $m\rightarrow 0$ pathwise on bounded time intervals in probability, without making strong boundedness assumptions on the coefficients $F, \gamma, \sigma$ and their derivatives.  These boundedness assumptions were made in the earlier work \cite{HMVW_14} and have also been made previously in other, related works \cite{freidlin2004,freidlin2011,sancho1982}.  From a physical standpoint, however, there are many natural model equations that do not satisfy these strong boundedness requirements, and therefore the use of the small-mass approximation of the dynamics above is in question.  Such an approximation has been instrumental in estimating chemical reaction rates \cite{kramers1940, smoluchowski1916}, simplifying computations of escape times from potential wells \cite{freidlin,weinan2010}, and answering ergodicity questions \cite{freidlin,weinan2010}.

To see the utility of our general result, we will apply it to three examples describing physically realizable dynamics, including the situation discussed in \cite{volpe2010} (see Section~3). In each of these examples, there is a confining force which grows unboundedly near the boundary of $\mathcal{X}$ (if it is non-empty) and/or near the point at infinity.  This unbounded force translates to, at the very least, unboundedness of some of the coefficients in the model equations \eqref{eqn:main}. Making use of our main result, we will be able to establish convergence in each of these three examples.

Compared with existing results on the small-mass limit with positive friction \cite{freidlin2004,freidlin2011,sancho1982,HMVW_14}, the hypotheses of our main result are extraordinarily weak.  Specifically, we only assume nominal regularity of the coefficients $F$, $\gamma$, $\sigma$ and that the believed limiting dynamics does not leave the set $\mathcal{X}$ in finite time.  It is worth emphasizing that we do not assume that the pair process defined by \eqref{eqn:main} also remains in the natural state space $\mathcal{X}\times \R^n$ for all finite times.  This makes our result more readily applicable because, while it is not always easy to control the family of exit times $\{ \tau_{\mathcal{X}\times \R^n}^m\}_{m>0}$ where $\tau_{\mathcal{X}\times \R^n}^m$ denotes the first exit time of $(x^m(t), v^m(t))$ from $\mathcal{X}\times \R^n$, it is more straightforward to control the exit time $\tau_\mathcal{X}$ of the limiting dynamics from $\mathcal{X}$.  An additional benefit of structuring the hypotheses in this way is that, as a consequence of our result, we gain control of the exit times $\tau_{\mathcal{X}\times \R^n}^m$ for $m>0$ small, in the sense we show that $\tau_{\mathcal{X}\times \R^n}^m \rightarrow \infty$ in probability as $m\rightarrow 0$.

The organization of this paper is as follows.  In Section~\ref{sec:mainres}, we state our main theoretical result (Theorem~\ref{thm:main}).  Section~\ref{sec:examples} gives a few physical, motivating examples for this work. In each example, we will verify that the hypotheses of Theorem~\ref{thm:main} are satisfied using the appropriate Lyapunov methods.  As a consequence of Theorem~\ref{thm:main}, we will therefore obtain the desired convergence as $m\rightarrow 0$ in each physical example studied. In Section~\ref{sec:proof}, we prove Theorem~\ref{thm:main}.

\section{Main Results}
\label{sec:mainres}

The limiting dynamics $x(t)$ will satisfy the It\^{o} stochastic differential equation  
\begin{align}
\label{eqn:limitsde}
dx(t) = [\gamma^{-1}(x(t)) F(x(t)) + S(x(t)) ] \, dt + \gamma^{-1}(x(t)) \sigma(x(t)) \, dB(t) 
\end{align}    
where, adopting the Einstein summation convention, the vector-valued function $S$ satisfies
\begin{align*}
S(x)= \left(\partial_{x_l} [\gamma^{-1}_{ij}(x)]J_{ jl}(x) \right)_{i=1}^n
\end{align*}
and the matrix $J$ solves the Lyapunov equation
\begin{align*}
J \gamma^T + \gamma J = \sigma \sigma^T .  
\end{align*}

To understand on some level how the equation \eqref{eqn:limitsde} could possibly define the limiting dynamics, we can try and formally set $m=0$ in equation~\eqref{eqn:main}, and solve for $v^0(t) \, dt = dx^0(t)$ using the second part of this equation.  This leads us to the following guess for the limiting equation
\begin{align*}
dx^0(t) = \gamma^{-1}(x^0(t))F(x^0(t)) \, dt + \gamma^{-1}(x^0(t)) \sigma(x^0(t)) \, dB(t) ,
\end{align*} 
where there is some ambiguity in how $\gamma^{-1}(x^0(t)) \sigma(x^0(t)) \, dB_t $ should be interpreted using the various conventions of stochastic integrals, e.g. It\^{o}, Stratonovich, Anti-It\^{o} \cite{karatzas, oksendal}.  The different conventions of the stochastic integral do not coincide because, even assuming $\sigma, \gamma$ are sufficiently smooth, as opposed to $\sigma(x^m(t))$ for $m>0$, $\gamma^{-1}(x^0(t)) \sigma(x^0(t))$ does not vary smoothly in $t$.  While one might suspect that the drift term $S(x(t))$ in equation~\eqref{eqn:limitsde} tells one how to interpret $ \gamma^{-1}(x^0(t)) \sigma(x^0(t)) \, dB_t $, this is not quite the case because there can be no relation between the type of stochastic integral and this drift in the most general case \cite{freidlin2011}.  Nevertheless this heuristic, first employed by Smoluchowski in \cite{smoluchowski1916} and later by Kramers in \cite{kramers1940}, serves as a good first step in understanding how some parts of \eqref{eqn:limitsde} arise.  See \cite{HMVW_14} for further, more specific details in how the noise-induced drift term, i.e. $S(x(t))$, in equation~\eqref{eqn:limitsde} is produced.   

Throughout the paper, we will make the following assumptions:
\begin{assumption}[Regularity, Positive-Definite Friction]
\label{assump:reg}
\label{assump:mf}
$F\in C^1( \mathcal{X}:\R^n)$, $\gamma\in C^2( \mathcal{X}: \R^{n\times n})$ and $\sigma \in C^1(\mathcal{X}:\R^{n\times k})$.  Moreover, for each $x\in \mathcal{X}$ the matrix $\gamma(x)$ is positive-definite; that is, for each $x\in \mathcal{X}$ and each $
y\in \R^n_{\neq 0}$ we have that $$(\gamma(x) y, y) >0.$$
\end{assumption}

\begin{assumption}[Non-explosivity of $x(t)$]
\label{assump:et}
The first exit time $\tau_\mathcal{X}$ of $x(t)$ from $\mathcal{X}$ is $\P$-almost surely infinite for all initial conditions $x_0=x\in \mathcal{X}$; that is, for all $x\in \mathcal{X}$
\begin{align*}
\P_x \{ \tau_\mathcal{X} < \infty\} =0.  
\end{align*}
\end{assumption} 

The regularity part of Assumption~\ref{assump:mf} assures that all equations in question make sense locally in time.  Critical to our main result is the positive-definite assumption made on the friction matrix $\gamma$.  This can be seen by taking a glance at equation~\eqref{eqn:limitsde}, for if the matrix $\gamma$ is simply non-negative we expect to get different behavior as $m\rightarrow 0$.  See \cite{freidlin2013} for an example of the small-mass limit when $\gamma$ vanishes on a set.  Assumption~\ref{assump:et} assures that the presumed limiting dynamics $x(t)$ remains in its domain of definition $\mathcal{X}$ for all finite times $t\geq 0$ almost surely.  Different from the previous references \cite{freidlin2004,HMVW_14, sancho1982}, we will not assume that the solution of \eqref{eqn:main} is non-explosive or, more importantly, that either $x(t)$ is contained in a compact subset of $\mathcal{X}$ or the coefficients $F, \gamma, \sigma$ are bounded on $\mathcal{X}$.  We do, however, need control over an additional derivative of $\gamma$.  Nevertheless, this should not be seen as an additional hypothesis, for this is a typical minimalist assumption needed to make sense of the pathwise solution of \eqref{eqn:limitsde} locally in time (see, for example, \cite{khasminskii, reybellet2006}).  An additional difference between our result and previous results is that we need not assume that $\gamma$ is uniformly positive definite on $\mathcal{X}$.  In some sense, however, the size of the smallest positive eigenvalue of $\gamma$ is controlled by non-explosivity (Assumption~\ref{assump:et}) of the solution of the limiting equation.            

Because we will not assume that the process $\{(x_t^m, v_t^m)\}_{t\geq 0}$, $m>0$, remains in $\mathcal{X}\times \R^n$ for all finite times, we will extend this process for times $t\geq \tau_{\mathcal{X}\times \R^n}^m$ where $\tau_{\mathcal{X} \times \R^n}^m$ is the first exit time of $(x_t^m, v_t^m)$ from $\mathcal{X}\times \R^n$.  In particular, letting $\Delta$ be some point not in $\R^n$, we set $x_t^m=v_t^m=\Delta$ for all times $t\geq   \tau_{\mathcal{X}\times \R^n}^m$.  To measure convergence of $x_t^m$ on the enlarged state space $\mathcal{X}\cup \{ \Delta\}$, let $d_\infty: \mathcal{X} \cup \{ \Delta \} \times \mathcal{X} \cup \{ \Delta\} \rightarrow [0, \infty]$ be given by 
\begin{align*}
d_\infty(x,y) = \begin{cases}
|x-y| & \text{ if } x,y \in \mathcal{X}\\
\infty & \text{ if } x=\Delta \text{ or } y = \Delta
\end{cases}.
\end{align*}
Observe that $d_\infty$ is not quite a metric since $d_\infty(\Delta, \Delta)=\infty$; however, $d_\infty$ satisfies the  remaining properties of a metric.  As we will see, it will serve us well as a slight generalization of a distance.  

We are now prepared to state our main results:

\begin{theorem}
\label{thm:main}
Suppose that Assumption~\ref{assump:mf} and Assumption~\ref{assump:et} are satisfied.  If the process $\{ x(t)\}_{t\geq 0}$ and the extended processes $\{x^m(t)\}_{t\geq 0}$ have the same initial condition $x\in \mathcal{X}$ for all $m>0$, then for every $T, \epsilon >0$
\begin{align*}
\P\bigg\{ \sup_{t\in [0, T]} d_\infty(x^m(t), x(t)) > \epsilon \bigg\} \rightarrow 0\,\,\text{ as }\,\, m \rightarrow 0. 
\end{align*}
\end{theorem}

\begin{remark}
To emphasize a remark made earlier, an aspect of the theorem that is particularly striking is that we make no explicit assumptions about the exit times $\{\tau^m_{\mathcal{X} \times \R^n}\}_{m>0}$ yet we still obtain pathwise convergence on compact time intervals in probability in $d_\infty$.  As we will see, not making this assumption about the exit times $\{\tau^m_{\mathcal{X} \times \R^n}\}_{m>0}$ is convenient because it is often easier to simply control $\tau_\mathcal{X}$.  Another interesting aspect of the result is that $d_\infty$ was constructed so that it penalizes the process $\{(x_t^m, v_t^m)\}_{t\geq 0}$ infinitely if it has exited $ \mathcal{X} \times \R^n$.  In particular, as a corollary of the proof of the theorem above, by having control over $\tau_\mathcal{X}$ we can obtain control over $\tau_{\mathcal{X}\times \R^n}^m$ for $m>0$, small.  
\end{remark}

\begin{corollary}
\label{cor:nolimexplosion}
Under the hypotheses of Theorem~\ref{thm:main}: For all $T>0$
\begin{align*}
\lim_{m\rightarrow 0}\P\{ \tau^m_{\mathcal{X}\times \R^n} \leq T \}=0.  
\end{align*}
In other words, $\tau_{\mathcal{X} \times \R^n}^m \rightarrow \infty$ in probability as $m\rightarrow 0$.  
\end{corollary}

\begin{remark}
Under the appropriate moment bounds and non-explosivity of the pair process $(x^m(t), v^m(t))$, one can apply Theorem~\ref{thm:main} to obtain stronger forms of convergence, e.g. convergence in $L^p$ for $p\geq 1$.    
\end{remark}

\section{Examples of Newtonian Dynamics with Unbounded Potentials}
\label{sec:examples}
In this section, we apply Theorem~\ref{thm:main} to physical examples realizable in a
laboratory. In the following, $x(t)$ will denote the position of one or more mesoscopic particles
in a liquid at a well-defined temperature $T$ (e.g. a Brownian particle coupled to a heat bath provided by the liquid, such as the ones experimentally studied in \cite{lanccon2001,volpe2010}). The particle is influenced by a force $F$, friction $\gamma$, and noise coefficient $\sigma$. For such a Brownian particle, the fluctuation-dissipation
relation holds:
\begin{equation}
\label{eq:FDrelation}
\gamma(x) \propto \sigma(x)\sigma^T(x).
\end{equation}

Although in each example there is a confining potential force which grows rapidly near the boundary $\partial \mathcal{X}$ and/or the point at infinity, it will be clear that Assumption~\ref{assump:reg} is satisfied.  Therefore, we will only need to see that Assumption~\ref{assump:et} is satisfied by showing that first exit time $\tau_\mathcal{X}$ of the limiting process $x(t)$ is almost surely infinite for all initial conditions $x\in \mathcal{X}$.  To show that $\tau_\mathcal{X}$ is almost surely infinite, we will use, by now standard, Lyapunov methods \cite{khasminskii,MTIII_93, reybellet2006}.  In particular, in each example we will exhibit a certain type of function $V\in C^2(\mathcal{X}: [0, \infty))$, called a Lyapunov function, which guarantees that 
\begin{align}
\label{eqn:noexp}
\P_x \{\tau_\mathcal{X} < \infty \}=0
\end{align}    
for all initial conditions $x\in \mathcal{X}$. To be more precise, define a sequence of open subsets $\mathcal{X}_k$, $k\in \N$, of $\mathcal{X}$ by 
\begin{align*}
\mathcal{X}_k = \begin{cases}
\{ x\in \mathcal{X} \, : \, \text{distance}(x, \partial \mathcal{X}) > k^{-1}\, \text{ and }\, |x| < k \} & \text{ if } \partial \mathcal{X} \neq \emptyset \\
\{ x\in \R^n \, : \, |x| < k \} & \text{ if } \partial \mathcal{X} = \emptyset
\end{cases}
\end{align*}
and observe that, if $\partial \mathcal{X}=\emptyset$, then $\mathcal{X}=\R^n$ as $\mathcal{X}$ is non-empty and both open and closed.  In each example we will exhibit a function $V\in C^2(\mathcal{X} : [0, \infty))$ satisfying the following two properties:
\begin{itemize}
\item[p1)]  There exists a sequence of positive constants satisfying $C_k \rightarrow \infty$ as $k\rightarrow \infty$ and  
\begin{align*}
V(x) \geq C_k \,\, \text{ for } x\in \mathcal{X}\setminus \mathcal{X}_k.
\end{align*}
\item[p2)]  There exist positive constants $C, D$ such that for all $x\in \mathcal{X}$
\begin{align*}
\mathcal{L}V(x) \leq C V(x) + D,
\end{align*}
where $\mathcal{L}$ denotes the infinitesimal generator of the Markov process $x(t)$. 
\end{itemize}
It follows that the existence of a function $V\in C^2(\mathcal{X}: [0, \infty))$ satisfying p1) and p2) above gives that $\P_x\{ \tau_{\mathcal{X}} < \infty\} =0$ for all $x\in \mathcal{X}$ (See, for example, Theorem~2.1 of \cite{MTIII_93}).

\subsection{Gravity and electrostatics}
We first prove convergence for the experimental example in
\cite{volpe2010} which originally motivated this work.  In \cite{volpe2010}, a Brownian
particle is in a vertical cylinder of finite height $b-a$ filled with water and the horizontal motion of the particle is assumed to be independent from its vertical motion. Therefore, $x(t)$ denotes the (one-dimensional) vertical position of the particle at time $t$ and the natural state space $\mathcal{X}$ of $x(t)$ is given by the open interval $(a,b)$ with $0\leq  a< b< \infty$.  The conservative forces acting on the particle are
given by the potential function
\begin{equation}
\label{eq:potfun}
U(x) ={B\over\kappa} e^{-\kappa (x-a)}+{B\over\kappa} e^{-\kappa (b-x)}+G_{\rm eff}x+ {e^{-\lambda (x-a)}\over (x-a)} + {e^{-\lambda (b-x)}\over (b-x
)}.
\end{equation}
The first two terms are due to double layer particle-wall forces, with $\kappa^{-1}$ the Debye length and $B>0$ a prefactor depending on the surface charge densities. The third term accounts for the effective gravitational contribution $G_{\rm eff} = \frac{4}{3}\pi R^3 (\rho_{\rm p} - \rho_{\rm s}) g$, with $g$ the gravitational acceleration constant, $R$ the radius of the particle, $\rho_{\rm p}$ the density of the particle and $\rho_{\rm s}$ the density of the fluid. Note that the the value of the first three terms of the potential at $x=a,b$ is finite but very large (as the prefactor $B$ is on the order of thousands of $k_{\rm B}T$); thus, to assure that the particle remains in the cylinder, the last two terms model ``soft walls'' at $x=a$ and $b$ and are fast-decaying away from the boundary with $\lambda \gg \kappa$. The forces are given by
\begin{equation*}
F(x) = -U'(x) 
\end{equation*}
and the friction coefficient is 
\begin{equation*}
\gamma(x) = \frac{k_BT}{D(x)},
\end{equation*}
where $D(x)$ is a hydrodynamic diffusion coefficient due to effects of particle wall
interactions. The exact form of $D$ is an infinite sum and can be found in \cite{brenner}. For our analysis, it is enough to know $D(x)\in C^2([a,b]:(0,\infty))$ with $D(a) = D(b) =0$, $D'(x)>0$ for $x\in[a,(a+b)/2)$, $D'(x)<0$ for $x\in((a+b)/2,b]$, and $D'((a+b)/2) = 0$.
Using the fluctuation-dissipation relation, the inertial system is given by
\begin{align*}
dx^m(t) =& v^m(t)\;dt\\
m\, dv^m(t) = &\bigg[F(x^m(t))-
\frac{k_BT}{D(x^m(t))}v^m(t)\bigg ]\;dt +\sqrt{ \frac{2(k_BT)^2}{D(x^m(t))}}\;dB(t),
\end{align*}
where $B(t)$ is a standard, one-dimensional Brownian motion.
The corresponding limiting equation is
\begin{align}
dx(t) =& \frac{F(x(t))D(x(t))}{k_BT}\;dt +
D'(x(t))dt +\sqrt{D(x(t))}\;dB(t).
\end{align}
To prove convergence of $x^m(t)$ to $x(t)$ in the sense described in Theorem~\ref{thm:main}, all we must show is that $\P_x(\tau_{(a,b)}=\infty)=1$ for all $x\in (a,b)$. To do so, we find the appropriate Lyapunov function as described at the beginning of this section. We define our candidate Lyapunov function to be the potential function $U$ and note that $U\in C^2((a,b) :[0, \infty))$ and, moreover, $U$ satisfies p1).  To see that p2) is satisfied, first apply the generator $\mathcal{L}$ of $x(t)$ to $U$ to find that 
\begin{equation}
\mathcal{L}U(x) = \left (-\frac{(U'(x))^2}{k_BT} + \frac{1}{2}U''(x)\right )D(x)+U'(x)D'(x),
\end{equation}
where we have replaced the force by $F(x) = -U'(x)$.  Because $x\mapsto \mathcal{L}U(x)$ is bounded  on every compact interval $[c,d]$ with $a<c$ and $d<b$, to produce the required estimate we focus on the behavior of this function near the endpoints $x=a,b$.  First fix $c\in (a, (a+b)/2)$.  Using the fact that $D(a)=0$, we can apply the mean value theorem to see that there exist constants $c_i>0$ such that for all $x\in (a, c]$
\begin{align*}
\bigg(\frac{(U'(x))^2}{k_B T}-\frac{1}{2}U''(x) \bigg)D(x)&=
 \bigg(\frac{(U'(x))^2}{k_B T}-\frac{1}{2}U''(x)\bigg) (D(x)-D(a))\\
 &=  \bigg(\frac{(U'(x))^2}{k_B T}-\frac{1}{2}U''(x)\bigg) D'(\xi_{x,a}) (x-a)\\
 &\geq \frac{c_1}{(x-a)^3}+ \frac{c_2}{(x-b)^4}  -c_3,  
\end{align*}
where $\xi_{x,a}$ is some point in $[a, c]$. 
By fixing $d\in ((a+b)/2, b)$ and adjusting the positive constants $c_i$ above, one can produce the same bound
\begin{align*}
\bigg(\frac{(U'(x))^2}{k_B T}-\frac{1}{2}U''(x) \bigg)D(x)&=
 \bigg(\frac{(U'(x))^2}{k_B T}-\frac{1}{2}U''(x)\bigg) (D(x)-D(b))\\
 &= \bigg(\frac{(U'(x))^2}{k_B T}-\frac{1}{2}U''(x)\bigg) (-D'(\eta_{x,b}) )(b-x)\\
 &\geq  \frac{c_1}{(x-a)^3} + \frac{c_2}{(x-b)^4} -c_3, 
 \end{align*}
 where $\eta_{x,b}$ is some point in $[d,b]$, which is satisfied for $x\in [d,b)$.  
Additionally, since $D'$ is bounded on $[a,b]$, there exists $C_1, C_2 >0$ such that 
\begin{align*}
|U'(x) D'(x)| \leq \frac{C_1}{(x-a)^2} + \frac{C_2}{(x-b)^2}
\end{align*} 
for all $x\in [a,b]$.  Putting these estimates together we find that $x\mapsto \mathcal{L} U(x)$ is bounded on $(a,b)$.  The bound in p2) then follows immediately.      

\subsection{1D interacting particles} We consider two close Brownian particles suspended in a fluid. If the separation between particles, denoted by $d$, is large enough that the Debye-H\"uckel linearization approximation can be made in the electrostatic potential of a system of ions in an electrolyte, then the DLVO theory \cite{derjaguin1941theory,verwey1999theory} gives the potential between colloidal spheres as
\begin{equation}
	\label{eq:DLVO}
U_{\rm DLVO}(d) = c
\frac{e^{- d/l}}{d},
\end{equation}
where the positive constants $c$ and $l$ depend on various properties of the two particles and $d$ is the separation distance of the particles. The diffusion coefficient $D=D(d)$ satisfies the following: $d\mapsto D(d)\in C^2([0,\infty):[0,\infty))$, $D(0) = 0$,
$D(d)\rightarrow D_{\mbox{SE}}<\infty$ as $d\rightarrow\infty$, $D'(d) >0$ and $D''(d)<0$ for all $0\leq d<\infty$. Additionally, the two particles are contained in a (common) shallow harmonic potential, $k x^2$, where $k$ is small compared to the constants in \eqref{eq:DLVO}. The particles' positions are described in one dimension using the potential function
\begin{equation}
\label{eq:potfun2}
U(x_1,x_2) = \frac{k}{2}(x_1^2+x_2^2) + U_{\rm DLVO}(x_2-x_1).
\end{equation}
Defining $d^m(t) = x_2^m(t)-x_1^m(t)$, the system is described by 
\begin{align*}
dx_i^m(t) = & v_i^m(t)dt\\
m\, dv_i^m(t) = &\Big [ -\partial_{x_i}U(x_1^m(t),x_2^m(t))-\frac{k_BT}{D(d^m(t))}v_i^m(t)\Big]\;dt+\sqrt{ \frac{2(k_BT)^2}{D(d^m(t))}}\;dB_i(t)
\end{align*}
where $i=1,2$, $B_1(t),B_2(t)$ are two standard, one-dimensional, independent Brownian motions.  The corresponding limiting equation is
\begin{align*}
dx_i(t) =& \left [-\partial_{x_i}U(x_1(t),x_2(t))\frac{D(d(t))}{k_BT} + (-1)^i D'(d(t))
\right ]\;dt + \sqrt{2D(d(t))}\;dB_i(t).  
\end{align*}
Here, the natural domain of definition for these processes is $\mathcal{X}\times \R^2$ and $\mathcal{X}$, respectively, where $$ \mathcal{X} = \{(x_1,x_2)\in\R^2: x_1<x_2\}.$$
To apply Theorem~\ref{thm:main}, we again need to see that $\P_x\{\tau_\mathcal{X}=\infty \}=1$ for all initial conditions $x=(x_1, x_2)\in \mathcal{X}$.  We define our candidate Lyapunov function to be the potential $U(x_1,x_2)$ as in \eqref{eq:potfun2} and now check to see that p1) and p2) are satisfied.  One can readily check that p1) is satisfied.  To see p2), apply the generator to $U$ to see that 
\begin{align}
\label{eq:generator1}
\mathcal{L}U(x_1,x_2) =& \left (-\frac{[(\partial_{x_1}U(x_1,x_2))^2+\partial_{x_2}U(x_1,x_2))^2]}{k_BT} + (\partial^2_{x_1}+\partial^2_{x_2})U(x_1,x_2)\right )D(x_2-x_1)\\
&\qquad +\Big[(\partial_{x_1}+\partial_{x_2})U(x_1,x_2)\Big ]D'(x_2-x_1). \nonumber
\end{align}
The partial derivatives above are given by 
\begin{align*}
	\partial_{x_i} U(x_1,x_2) =& k x_i + (-1)^{i-1}\frac{c e^{-(x_2-x_1)/l}}{(x_2-x_1)}\left(\frac{1}{l} + \frac{1}{(x_2-x_1)}\right )\\
	\partial_{x_i}^2 U(x_1,x_2) =& k +\frac{c e^{-(x_2-x_1)/l }}{(x_2-x_1)}\left (\frac{1}{l^2 } + \frac{2}{l (x_2-x_1)}+\frac{2}{(x_2-x_1)^2}\right )
\end{align*} 
and 
\begin{equation*}
	(\partial_{x_1}+\partial_{x_2})U(x_1,x_2) = k(x_1+x_2). 
\end{equation*} 
Using the mean value theorem and the fact that $D(0)=0$, there exist constants $c_i>0$ and $\xi_{x_1,x_2} \geq 0$ such that  
\begin{align*}
&\left (\frac{[(\partial_{x_1}U(x_1,x_2))^2+\partial_{x_2}U(x_1,x_2))^2]}{k_BT} - (\partial^2_{x_1}+\partial^2_{x_2})U(x_1,x_2)\right )D(x_2-x_1)\\
&= \left (\frac{[(\partial_{x_1}U(x_1,x_2))^2+\partial_{x_2}U(x_1,x_2))^2]}{k_BT} - (\partial^2_{x_1}+\partial^2_{x_2})U(x_1,x_2)\right )D'(\xi_{x_1,x_2})(x_2-x_1) \\
&\geq - c_1(x_1^2+x_2^2) - c_2
\end{align*}
for all $(x_1, x_2) \in \mathcal{X}$.  In the estimate above, we have used the facts that 
\begin{align*}
\sup_{\xi \geq 0} D'(\xi) \in (0, \infty) \,\,\text{ and } \,\, \inf_{\xi \in [0, c]} D'(\xi) \in (0, \infty)
\end{align*}
for all $c>0$, as $D''(\xi) < 0$ for $\xi \geq 0$.  Combining the above estimate with the bound
\begin{align*}
|(\partial_{x_1} + \partial_{x_2}) U(x_1, x_2) D'(\xi)| \leq  k D_{\min}' (|x_1|+|x_2|),
\end{align*} 
which is satisfied for all $\xi \in [0, \infty)$ and all $(x_1, x_2) \in \mathcal{X}$, produces the required estimate p2). 

\subsection{Non-conservative forces}
In the previous two examples, one can easily adapt the arguments given there to show that the pair process $(x^m(t),v^m(t))$ never leaves $\mathcal{X}\times \R^n$ for each $m>0$ by simply taking $U(x) + \frac{1}{2}mv^2$ to be our candidate Lyapunov function.  
In this example, we introduce non-conservative forces in a 2D system where finding a
Lyapunov function for the system when $m>0$ is difficult.   This is because there is no
potential function for all of the external forces. 

Adding the rotational force field to the Langevin equations for the Brownian motion of a particle in the $(x_1,x_2)$-plane, the corresponding non-conservative
forces are:
\begin{equation}
\left\{\begin{array}{ccc}
\displaystyle \hat{F}_{x_1}(x_1,x_2) & = & - \gamma\Omega x_2 \\ [12pt]
\displaystyle \hat{F}_{x_2}(x_1,x_2) & = & + \gamma\Omega x_1
\end{array}\right.
\end{equation}
The terms $-\gamma\Omega x_2$ and $+\gamma\Omega x_1$ introduce a coupling between the equations, which becomes apparent in the fact that the cross-correlation is non-zero. This can in fact be realized experimentally by, for example, using transfer of orbital angular momentum to an optically trapped particle \cite{volpe2006torque}.  In addition to the non-conservative forces, the particle is confined to a pore, i.e. a well with radius $C$ centered at $(x_1, x_2)=(0,0)$.  We now define the radially symmetric potential $U(x_1, x_2)$ and the diffusion gradient.  We assume that $U(x_1, x_2)= \mathcal{U}(r^2(x_1,x_2))$ where $r^2(x_1, x_2)=x_1^2+x_2^2$ and $\mathcal{U}\in C^2([0, C^2): [0, \infty))$ satisfies 
\begin{align*}
\mathcal{U}(r) = {B\over\kappa (C^2-r)} e^{-\kappa (C^2-r)}
\end{align*}  
for $r\in [0, C^2)$.  The diffusion gradient is such that for $r\in [0, C)$, $D(r)=\mathcal{D}(r^2)$ where $\mathcal{D}\in C^2([0, C^2]: [0, \infty))$ satisfies $\mathcal{D}(C^2) = 0$ and $\mathcal{D}(r)<D_{\rm SE}$, $-\infty<\mathcal{D}'(r)<0$, $\mathcal{D}''(r)<0$ for $0\leq r\leq C^2$.
Define $r^m(t)^2 ={x_1^m(t)^2+x_2^m(t)^2}$. The full system then becomes
\begin{equation*}
\label{eq:Ex3sys}
\left \{ \begin{array}{rcl}
dx_i^m(t) &=& v_i^m(t)\;dt\\
m\, dv^m_i(t) &= & \left [-(\partial_{x_i} U)(x_1^m(t), x_2^m(t))
-\frac{k_BT}{D(r^m(t))}\Omega x^m_{j}(t)-\frac{k_BT}{D(r^m(t))}v_i^m(t)\right]\;dt + \sqrt{\frac{2(k_BT)^2}{D(r^m(t))}}\;dB_i(t),
\end{array}\right .
\end{equation*}
$i=1,2$, $j\neq i$, and where $B_i(t)$ are two standard, one-dimensional, independent Brownian motions.  The corresponding limiting equation is
\begin{align}
\label{eqn:ex3ld}
dx_i(t) =& \left [-\mathcal{U}'(r^2(t))\frac{2x_i(t)\mathcal{D}(r^2(t))}{ k_BT}-\Omega x_j(t)+2x_i(t)\mathcal{D}'(r^2(t))\right ]\;dt + \sqrt{2 \mathcal{D}(r^2(t))}\;dB_i(t). 
\end{align}
A suitable choice of a Lyapunov function for the dynamics \eqref{eqn:ex3ld} is the potential function $U(x_1, x_2)$. This choice works intuitively because the non-conservative forces are bounded inside the pore and are dominated by the potential function near the boundary.  To see that this intuition is indeed true, note that p1) is clearly satisfied.  To see p2), apply the generator of the process $(x_1(t), x_2(t))$ to $U(x_1, x_2)$ to find that 
\begin{align}
\label{eq:generator}
\mathcal{L}U(x_1, x_2) =& -\frac{4r^2\left (\mathcal{U}'(r^2)\right)^2\mathcal{D}(r^2)}{k_BT} 
+ 4\left ( r^2\mathcal{U}''(r^2)+\mathcal{U}'(r^2)\right )\mathcal{D}(r^2)\\
-&4\Omega x_1x_2\mathcal{U}'(r^2) +4r^2\mathcal{U}'(r^2)\mathcal{D}'(r^2).\nonumber
\end{align}
By assumption $\mathcal{D}(C^2)=0$, and $\mathcal{D}(r)$, $\mathcal{D}'(r)$ are bounded on $[0, C^2]$.  Moreover, $\mathcal{D}'(r)<0$ on $[0, C^2]$.  Thus using  the mean value theorem, we find that there exist constants $c_i >0$ such that 
\begin{align*}
	&\left (\frac{4r^2\left (\mathcal{U}'(r^2)\right)^2}{k_BT} 
- 4\left (r^2\mathcal{U}''(r^2)+\mathcal{U}'(r^2)\right )\right )\mathcal{D}(r^2) \\
&=\left (\frac{4r^2\left (\mathcal{U}'(r^2)\right)^2}{k_BT} 
- 4\left (r^2\mathcal{U}''(r^2)+\mathcal{U}'(r^2)\right )\right )(\mathcal{D}(r^2)-\mathcal{D}(C^2))\\
	&\geq \frac{c_1}{(C^2-r^2)^3}-c_2.
\end{align*}
Additionally since $\mathcal{D}'$ is bounded on $[0, C^2]$, there exists $C_1,C_2>0$ such that
\begin{equation*}
	\left |-2\Omega x_1x_2U'(r^2) + 4r^2U'(r^2)D'(r^2)\right | \leq \frac{C_1}{(C^2-r^2)^2}+C_2. 
\end{equation*}
Putting these estimates together we find that $x\mapsto \mathcal{L} U(x)$ is bounded on $\mathcal{X}$.  Property p2) now follows easily.

\section{Proof of Main Result}
\label{sec:proof}
In this section we prove Theorem~\ref{thm:main} and Corollary~\ref{cor:nolimexplosion}.  The idea underlying the proof of both results is quite natural.  First we will see that due to the structure of equation~\eqref{eqn:main}, for each $m>0$ the first exit time $\tau_\mathcal{X}^m$ of $x^m(t)$ from $\mathcal{X}$ coincides with $\tau_{\mathcal{X}\times \R^n}^m$.  In other words if the process $(x^m(t), v^m(t))$ exits the domain $\mathcal{X} \times \R^n$, then $x^m(t)$ must have exited $\mathcal{X}$.  Once we have control of the stopping times in this way, the goal is to then construct processes $\{ x_k(t)\}_{t \geq 0}$ and $\{ x_k^m(t)\}_{t\geq 0}$, $m>0$ and $k\in \N$, on $(\Omega, \mathcal{F}, \P)$ satisfying the following properties:
\begin{enumerate}
\item $\{ x_k^m(t)\}_{t\geq 0}\subseteq \R^n$ and $\{ x_k(t)\}_{t \geq 0}\subseteq \R^n$, i.e., all processes live in the ambient space $\R^n$ (as opposed to $\mathcal{X}$) for all finite times $t\geq 0$.
\item $\{x_k^m(t)\}_{t\geq 0}$ and $\{ x_k(t)\}_{t\geq 0}$ have continuous sample paths.  
\item  Letting 
\begin{align*}
\mathcal{X}_k = \begin{cases}
\{ x\in \mathcal{X} \, : \, \text{distance}(x, \partial \mathcal{X}) > k^{-1} \text{ or } |x| < k \} & \text{ if } \partial \mathcal{X} \neq \emptyset \\
\{ x\in \mathcal{X}=\R^n \, : \, |x| < k \} & \text{ if } \partial \mathcal{X} = \emptyset
\end{cases}
\end{align*}
and $\tau_{\mathcal{X}_k}^m$, $\tau_{\mathcal{X}_k}$ denote the first exit times of, respectively, $x^m(t)$ and $x(t)$ from $\mathcal{X}_k$:
\begin{align*}
x_k^m(t) \equiv x^m(t)\,\, \text{ for } \,\, 0\leq t< \tau_{\mathcal{X}_k}^m\,\, \text{ and }\,\, x_k(t) \equiv x(t)\,\, \text{ for }0\leq t < \tau_{\mathcal{X}_k} \qquad \P-\text{almost surely}.
\end{align*}
In the definition of $\mathcal{X}_k$ above, we note that if $\partial \mathcal{X}=\emptyset$ then $\mathcal{X}=\R^n$, as $\mathcal{X}$ is non-empty and both open and closed.  
\item For every $\epsilon, T>0$, $k\in \N$ and $x_k^m(0)=x_k(0)=x\in \mathcal{X}$
\begin{align*}
\lim_{m\rightarrow 0}\P\Big\{ \sup_{t\in [0, T]} |x_k^m(t) - x_k(t)| > \epsilon\Big\} =0.  
\end{align*}
\end{enumerate}
The processes $\{ x_k^m(t) \}_{t\geq 0}$ and $\{ x_k(t) \}_{t\geq 0}$ should be thought of as localizations (in time) of our original processes $\{ x^m(t)\}_{t\geq 0}$ and $\{ x(t) \}_{t\geq 0}$ which satisfy the desired convergence as $m\rightarrow 0$ for each $k\in \N$.  Formally taking $k\rightarrow \infty$ and exchanging the order of limits in (4) above we may expect on an intuitive level the convergence to hold.  However, performing such an exchange is nontrivial.  Nevertheless, due to the way the set $\,\mathcal{X}\,$ is stratified by $\{\mathcal{X}_k\}_{k\in \N}$, we will see at the end of this section that this intuition is indeed correct; that is, we can extract convergence given the existence of such approximate processes.  Corollary~\ref{cor:nolimexplosion} will be an easy consequence of the proof of Theorem~\ref{thm:main}.

We begin the section by showing $\tau_\mathcal{X}^m=\tau_{\mathcal{X}\times \R^n}^m$ almost surely for all $m>0$ and by constructing the approximate processes satisfying (1)-(4) above.  Afterwards, we will prove Theorem~\ref{thm:main} and Corollary~\ref{cor:nolimexplosion}.

\begin{lemma}
\label{lem:tc}
Suppose that Assumption~\ref{assump:mf} is satisfied.  Then for each $m>0$ and each initial condition $(x, v)\in \mathcal{X}\times \R^n$ $$\P\{\tau_\mathcal{X}^m= \tau_{\mathcal{X}\times \R^n}^m\}=1.$$  
Moreover, there exist processes $\{ x_k^m(t)\}_{t\geq 0}$ and $\{x_k(t) \}_{t\geq 0}$, $k\in \N$ and $m>0$, on the probability space $(\Omega, \mathcal{F}, \P)$ satisfying properties (1)-(4) above.  
\end{lemma}

\begin{proof}[Proof of Lemma~\ref{lem:tc}]
We will start by constructing the desired family of processes $\{ x_k^m(t)\}_{t\geq 0}$ and $\{ x_k(t)\}_{t\geq 0}$.  The conclusion $\tau_\mathcal{X}^m=\tau_{\mathcal{X}\times \R^n}^m$ $\P$-almost surely will be shown in the process of constructing these approximations.  

By the existence of smooth bump functions, for each $k\in \N$ there exists $g_k \in C^\infty(\R^n \, : \,[0, 1])$ satisfying
\begin{align*}
g_k(x) = \begin{cases}
1 & \text{ if } x\in \overline{\mathcal{X}}_k\\
0 & \text{ if } x\in \R^n \setminus \mathcal{X}_{k +1}
\end{cases}.
\end{align*} 
Let $\hat{F}: \R^n \rightarrow \R^n$, $\hat{\sigma}: \R^n \rightarrow \R^{n\times k}$ be $C^\infty$ and have bounded derivatives of all orders, and let $\hat{\gamma}= c \,\text{Id}_{n\times n}$ where $\text{Id}_{n\times n}$ is the $n\times n$ identity matrix and $c>0$ is a fixed, arbitrary constant.  For each $k\in \N$, define $F_k, \sigma_k, \gamma_k$ on $\R^n$ by 
 \begin{align*}
 F_k= g_k F + (1-g_k) \hat{F}, \,\,\,\, \sigma_k = g_k \sigma + (1-g_k)\hat{\sigma}, \,\, \,\, \gamma_k = g_k \gamma + (1-g_k) \hat{\gamma}.
 \end{align*} 
 By construction, observe that $F_k, \gamma_k, \sigma_k$ are bounded and globally Lipschitz on $\R^n$.  Also, letting 
  \begin{align*}
 c_{k}:=\inf_{\substack{x\in \overline{\mathcal{X}}_{k+1}\\y\in \R^n_{\neq 0}}} \frac{(\gamma(x) y, y)}{|y|^2},
 \end{align*}
 we note that $c_k>0$ as $\overline{\mathcal{X}}_{k+1}$ is compact.  Moreover, $\gamma_k \in C^2(\R^n : \R^{n\times n})$ is uniformly positive definite on $\R^n$ since 
 \begin{align*}
 (\gamma_k(x) y, y) &= g_k(x)( \gamma(x) y, y) + (1-g_k(x))c |y|^2 \\
 &\geq g_k(x) c_k|y|^2 + (1-g_k(x)) c|y|^2 \\
 & \geq\min\{ c_k, c\} |y|^2 .   
 \end{align*} 
Now consider the family of $\R^n \times \R^n$-valued
SDEs 
\begin{align}
\label{eqn:SDEapprox}
d x_k^m(t) & = v_k^m(t) \, dt\\
\nonumber m \, dv^m_k(t)& = [F_{k}(x_k^m(t)) - \gamma_k (x_k^m(t)) v_k^m(t) ] \, dt  + \sigma_k(x_k^m(t)) \, dB_t  
\end{align} 
indexed by the parameters $k\in \N$ and $m>0$ and the family of $\R^n$-valued SDEs given by 
\begin{align*}
 dx_k(t) = [\gamma_k^{-1}(x_k(t)) F_k(x_k(t)) -S_k(x_k(t))] \, dt + \gamma_{k}^{-1}(x_k(t)) \sigma_k(x_k(t)) \, dB_t, 
 \end{align*} 
where $S_k$ is the noise-induced drift term determined by $\gamma_k,\sigma_k$.

We now show that $\{ (x_k^m(t), v_k^m(t))\}_{t\geq 0} \subset \R^n \times \R^n$.  By construction, we saw that the coefficients $F_k, \gamma_k, \sigma_k$ are bounded and globally Lipschitz on $\R^n$.  However, the SDE \eqref{eqn:SDEapprox} has only locally Lipschitz coefficients as the term $\gamma_k(x)v$ is a locally Lipschitz function on $\R^n \times \R^n$.  Therefore to see that  $\{ (x_k^m(t), v_k^m(t))\}_{t\geq 0} \subset \R^n \times \R^n$ we construct the appropriate Lyapunov functions.  Pick $h\in C^\infty(\R^n : [0, \infty))$ to satisfy the following two properties: 
\begin{itemize}
\item[(a)] $h(x) \rightarrow \infty$ as $|x|\rightarrow \infty$.
\item[(b)] For each $j=1,\ldots, n$, $\partial_{x_j} h$ is a bounded function on $\R^n$.     
\end{itemize}
Define $\Phi(x,v)= h(x) + |v|^2$ and let $\mathcal{L}_{k}^m$ denote the infinitesimal generator of the Markov process defined by \eqref{eqn:SDEapprox}.  By construction and uniform positivity of the matrix $\gamma_k$, it is not hard to check that for each $m>0$, $k\in \N$ fixed  
\begin{align*}
(x,v) \mapsto \mathcal{L}_k^m \Phi(x, v)
\end{align*}
 is bounded on $\R^n \times \R^n$.  It now follows easily by the standard Lyapunov function theory (see, for example, \cite{reybellet2006}): for each fixed $m>0, k\in \N$ we have that $\{ (x_k^m(t), v_k^m(t))\}_{t\geq 0} \subset \R^n \times \R^n$ almost surely.    

Before verifying that the remaining properties in (1)-(4) are satisfied, let us take a moment to see that $\P \{ \tau_\mathcal{X}^m = \tau_{\mathcal{X} \times \R^n}^m \} =1$ for all $m>0$ and all initial conditions $(x,v)\in \mathcal{X}\times \R^n$.  Trivially, $\tau_{\mathcal{X}\times \R^n}^m\leq \tau_{\mathcal{X}}^m$ almost surely. Next we prove the opposite inequality. Let $\xi_l^m= \inf\{t>0\,: |v^m(t)| \geq l \}$.  Then for all $j, l \in \N$ and all $m>0$, we have the almost sure inequality
\begin{align*}
\tau_{\mathcal{X}_j}^m \wedge \xi_{l}^m \leq \tau_{\mathcal{X}\times \R^n}^m . 
\end{align*}  
The goal is to show that for all $j\in \N$
\begin{align}
\label{eqn:limitst}
\lim_{l\rightarrow \infty} \tau_{\mathcal{X}_j}^m \wedge \xi_{l}^m = \tau_{\mathcal{X}_j}^m  \leq \tau_{\mathcal{X}\times \R^n}^m 
\end{align}  
almost surely.  Taking $j\rightarrow \infty$ in the expression above will then establish the desired conclusion.  By construction and pathwise uniqueness, if $(x_k^m(0), v_k^m(0))=(x^m(0), v^m(0))=(x,v) \in \mathcal{X}\times \R^n$, then 
 \begin{align*}
 \P \bigg\{ \sup_{t\in [0, \tau_{\mathcal{X}_k}^m)} |(x_k^m(t), v_k^m(t))-(x^m(t), v^m(t))| =0\bigg\}=1
 \end{align*}
 as the coefficients defining both pair processes agree on $\mathcal{X}_{k}\times \R^n$.  In particular, we have established \eqref{eqn:limitst} as $v_k^m(t)$, hence $v^m(t)$, has yet to exit $\R^n$ before time $\tau_{\mathcal{X}_k}^m$.

Now we turn our attention to showing the remaining properties in the list (1)-(4).  To see that property (1) is satisfied,  we have already seen using the Lyapunov function $\Phi(x,v)$ that $\{ x_k^m(t)\}_{t\geq 0}\subseteq \R^n$ for all $m>0$ and $k\in \N$.  To see that $\{ x_k(t) \}_{t\geq 0}\subseteq \R^n$ for all $k\in \N$, by construction, the coefficients $\gamma_k^{-1} F_k$ and $\gamma_k^{-1} \sigma_k$ are globally Lipschitz functions on $\R^n$.  This can be seen using the derivative formula
\begin{align*}
\frac{\partial \gamma_k^{-1}}{\partial x_l} (x)  = - \gamma_k^{-1}(x)  \frac{\partial \gamma_k^{-1}}{\partial x_l}(x)\gamma_k^{-1}(x) 
\end{align*}
and using the fact that all matrices on the right above are bounded on $\R^n$.  Also, one can use the fact that the unique solution $J_k$ of the Lyapunov equation $J_k \gamma_k^T+ \gamma_k J_k = \sigma_k \sigma_k^T$ is given by (see \cite{HMVW_14, ortega}) 
 \begin{align*} 
 J_k(x)= -\int_0^\infty \exp(- t \gamma_k(x)) \sigma_k(x) \sigma^T_k(x) \exp(-t \gamma_k^T(x)) \, dt 
 \end{align*}
to see that, too, $S_k$ is globally Lipschitz.  By the standard pathwise existence and uniqueness theorem for solutions of SDEs, we now see that $\{ x_k(t)\}_{t\geq 0}\subseteq \R^n$.  

Properties (2) and (3) follow immediately by construction.  To obtain property (4), apply Theorem~ 1 of \cite{HMVW_14}.      
\end{proof}
We now have all of the tools necessary to prove Theorem~\ref{thm:main} and Corollary~\ref{cor:nolimexplosion}.    
\begin{proof}[Proof of Theorem~\ref{thm:main}]
For any $T, \epsilon, m>0$ we have that 
\begin{align*}
\nonumber \P\Big\{ \sup_{t\in [0,T]} d_\infty(x^m(t),x(t))>\epsilon \Big\}&=\P \Big\{\sup_{t\in [0,T]} d_\infty(x^m(t), x(t))>\epsilon, \, \tau_{\mathcal{X} \times \R^n}^m \leq T \Big\}\\
&= \P \{ \tau_{\mathcal{X} \times \R^n}^m \leq T\} + \P  \Big\{\sup_{t\in [0,T]} |x^m(t)-x(t)|>\epsilon, \, \tau_{\mathcal{X} \times \R^n}^m > T \Big\}
\end{align*}
where on the last line we used the fact that $x(t) \in \mathcal{X}$ for all finite times $t\geq 0$ almost surely.  Applying Lemma~\ref{lem:tc}, we see that for any $T, \epsilon, m >0$ and any $k\in \N$ 
\begin{align*}
\nonumber \P\Big\{ \sup_{t\in [0,T]} d_\infty(x^m(t),x(t))>\epsilon \Big\}&=\P \{ \tau_{\mathcal{X}}^m \leq T\} + \P  \Big\{\sup_{t\in [0,T]} |x^m(t)-x(t)|>\epsilon, \, \tau_{\mathcal{X}}^m > T \Big\}\\
&\leq 2 \P \{ \tau_{\mathcal{X}_k}^m \wedge \tau_{\mathcal{X}_k} \leq T\} + \P \Big\{\sup_{t\in [0,T]} |x^m(t)- x(t)| >\epsilon, \, \tau_{\mathcal{X}_k}^m \wedge \tau_{\mathcal{X}_k} > T\Big\}\\
&\leq 2 \P \{ \tau_{\mathcal{X}_k}^m \wedge \tau_{\mathcal{X}_k}\leq T\} + \P \Big\{\sup_{t\in [0,T]} |x^m_k(t)- x^m_k(t)| >\epsilon\Big\}
\end{align*}
where the first inequality was obtained by partitioning each event $A$ in question as
\begin{align*}
A=( A\cap \{ \tau_{\mathcal{X}_k}^m \wedge \tau_{\mathcal{X}_k} \leq T\}) \cup (A\cap \{\tau_{\mathcal{X}_k}^m \wedge \tau_{\mathcal{X}_k} > T\})
\end{align*}
and estimating their associated probabilities by containment.  By property (4), for each $\epsilon>0$ and $k\in \N$: 
\begin{align*}
\P \Big\{ \sup_{t\in [0,T]} |x^m_k(t)- x_k(t)|>\epsilon \Big\} \rightarrow 0 \,\, \text{ as } \,\, m \rightarrow 0
\end{align*}
so we turn to bounding $\P\big\{\tau_{\mathcal{X}_k}^m \wedge \tau_{\mathcal{X}_k} \leq T\big\}$.  Notice that
\begin{align*}
\P\big\{\tau_{\mathcal{X}_k}^m \wedge \tau_{\mathcal{X}_k} \leq T\big\}& \leq \P\Big\{\tau_{\mathcal{X}_k}^m \wedge \tau_{\mathcal{X}_k} \leq T, \sup_{t\in [0,T]} |x^m_{k+1}(t) - x_{k+1}(t)| \leq \epsilon \bigg\}\\
&\qquad +\P \Big\{\sup_{t\in [0,T]} |x_{k+1}^m(t)- x_{k+1}(t)|>\epsilon  \Big\}. 
\end{align*}
Because we have control of the latter term on the last line above, the crucial observation is that 
for all $\epsilon \in (0,1/2)$, $k\geq 2$ $$\Big\{ \tau_{\mathcal{X}_k}^m \wedge \tau_{\mathcal{X}_k}  \leq T,\, \sup_{t\in [0,T]} |x_{k+1}^m(t) -x_{k+1}(t)|\leq \epsilon  \Big\}\subset \{ \tau_{\mathcal{X}_{N(\epsilon, k)}}\leq T\}$$ 
for some integer $N(\epsilon, k)\geq 1$ satisfying $\lim_{k \rightarrow \infty} N(\epsilon, k)=N(\epsilon)\in \N\cup\{\infty \}$ and, if $N(\epsilon)< \infty$, $\lim_{\epsilon \rightarrow \infty} N(\epsilon) =\infty$.  Putting this all together we obtain the following estimate for all $ m, T>0$, all $\epsilon \in (0,1/2)$, $k\geq 2$  
\begin{align*}
\P\Big\{ \sup_{t\in [0,T]} d_\infty(x^m(t),x(t))>\epsilon \Big\}&\leq 2\P\{\tau_{\mathcal{X}_{N(\epsilon, k)}} \leq T\} + 2 \P \Big\{\sup_{t\in [0,T]} |x_{k+1}^m(t)- x_{k+1}(t)|>\epsilon  \Big\}\\
&\qquad +  \P \Big\{ \sup_{t\in [0,T]} |x^m_k(t)- x_k(t)|>\epsilon \Big\}. 
\end{align*}
Thus for all $T>0$, $\epsilon \in (0,1/2)$, $k\geq 2$ we have 
\begin{align*}
\limsup_{m\rightarrow 0} \P\Big\{ \sup_{t\in [0,T]} d_\infty(x^m(t),x(t))>\epsilon \Big\}&\leq 2\P\{\tau_{\mathcal{X}_{N(\epsilon, k)}} \leq T\}.  
\end{align*}
Taking $k\rightarrow \infty$ in the above we obtain the following inequality  
\begin{align*}
\limsup_{m\rightarrow 0} \P\Big\{ \sup_{t\in [0,T]} d_\infty(x^m(t),x(t))>\epsilon \Big\}&\leq \begin{cases} 2\P\{\tau_{\mathcal{X}_{N(\epsilon)}} \leq T\} & \text{ if } N(\epsilon) \in \N\\
0 & \text{ otherwise}  \end{cases}
\end{align*}
for all $\epsilon \in (0,1/2)$.  In particular, the result is proven in the case when $N(\epsilon)=\infty$.  If $N(\epsilon) \in \N$, then for $\delta \in (0, \epsilon)$, $\epsilon < 1/2$ we have  
\begin{align*}
\limsup_{m\rightarrow 0} \P\Big\{ \sup_{t\in [0,T]} d_\infty(x^m(t),x(t))>\epsilon \Big\} \leq 2\P\{\tau_{\mathcal{X}_{N(\delta)}} \leq T\}.  
\end{align*}
Taking $\delta \downarrow 0$, using the fact that $N(\delta) \rightarrow \infty$ and the fact that $\tau_{\mathcal{X}}= \infty $ almost surely, we obtain the result.  
 
\end{proof}

\begin{proof}[Proof of Corollary~\ref{cor:nolimexplosion}]
This follows easily by Theorem~\ref{thm:main} since we have already seen that 
\begin{align*}
\nonumber \P\Big\{ \sup_{t\in [0,T]} d_\infty(x^m(t),x(t))>\epsilon \Big\}&=\P \{ \tau_{\mathcal{X} \times \R^n}^m \leq T\} + \P  \Big\{\sup_{t\in [0,T]} |x^m(t)-x(t)|>\epsilon, \, \tau_{\mathcal{X} \times \R^n}^m > T \Big\}.  
\end{align*}
\end{proof}

\bibliographystyle{plain}
\bibliography{smaunbHerzog}

\end{document}